\documentclass[11pt,a4paper,twoside]{amsart}

\usepackage{amssymb} \usepackage{mathtools} \usepackage{graphics}
\usepackage{bm} \usepackage{tikz} 

\newtheorem{theo}{\bf Theorem}[section]  \newtheorem{lemma}[theo]{\bf Lemma}
\newtheorem{conj}[theo]{\bf Conjecture} \newtheorem{defi}[theo]{\bf
Definition} \newtheorem{coro}[theo]{\bf Corollary}
 \theoremstyle{remark}
\newtheorem{remark}[theo]{\bf Remark}

 \newcommand{\mbf}{\mathbf}
 \newcommand{\Z}{{\mathbb Z}}
\newcommand{\N}{{\mathbb N}} 
 
\newcommand{\la}{\leftarrow} \newcommand{\ra}{\rightarrow}

\newcommand{\exA} {
	\begin{tikzpicture} \draw(0.5,2.5)--(10.5,2.5);
\node[circle,draw=black] at (1,2){$2$}; \node[circle,draw=black] at
(2,2){$1$}; \node[circle,draw=black] at (3,2){$3$}; \node[circle] at
(4,2){$\infty$}; \node[circle,draw=black] at (5,2){$2$};
\node[circle,draw=black] at (6,2){$4$}; \node[circle,draw=black] at
(7,2){$3$}; \node[circle,draw=black] at (8,2){$4$}; \node[circle] at
(9,2){$\infty$}; \node[circle,draw=black] at (10,2){$1$};
                \node at (1,3){$2$}; \node at (2,3){$1$};
\node at (3,3){$\infty$}; \node at (4,3){$2$};
\node at (5,3){$\infty$}; \node at (6,3){$\infty$};
\node at (7,3){$3$}; \node at (8,3){$\infty$};
\node at (9,3){$3$}; \node at (10,3){$1$};
	\end{tikzpicture} }

\newcommand{\exB} {
	\begin{tikzpicture} \draw(0.5,2.5)--(10.5,2.5); \node[circle]
at (1,2){$\infty$}; \node[circle,draw=black] at (2,2){$1$};
\node[circle,draw=black] at (3,2){$2$}; \node[circle] at
(4,2){$\infty$}; \node[circle,draw=black] at (5,2){$2$};
\node[circle,draw=black] at (6,2){$3$}; \node[circle,draw=black] at
(7,2){$3$}; \node[circle,draw=black] at (8,2){$4$}; \node[circle] at
(9,2){$\infty$}; \node[circle,draw=black] at(10,2){$1$};
 \node at (1,3){$2$}; \node at (2,3){$1$};
\node at (3,3){$\infty$}; \node at (4,3){$2$};
\node at (5,3){$\infty$}; \node at (6,3){$\infty$};
\node at (7,3){$3$}; \node at (8,3){$\infty$};
\node at (9,3){$3$}; \node at (10,3){$1$};
	\end{tikzpicture} }

\newcommand{\exC} { 
	\begin{tikzpicture} \draw(0.5,1.5)--(9.5,1.5);

		\node at (1,2){$2$};
\node at (2,2){$1$}; \node at
(3,2){$4$}; \node at (4,2){$3$};
\node at (5,2){$4$}; \node at
(6,2){$4$}; \node at (7,2){$1$};
\node at (8,2){$1$}; \node at
(9,2){$2$};

   	\node[circle,draw=black] at (1,1){$2$}; \node[circle] at
(2,1){{\color{black}$1$}}; \node[circle,draw=black] at (3,1){$1$};
\node[circle,draw=black] at (4,1){$2$}; \node[circle,draw=black] at
(5,1){$3$}; \node[circle,draw=black] at (6,1){$4$};
\node[circle,draw=black] at (7,1){$1$}; \node[circle] at
(8,1){{\color{black}$1$}}; \node[circle,draw=black] at (9,1){$1$};
    
	\end{tikzpicture} } \newcommand{\exD} { 
	\begin{tikzpicture} \draw(0.5,1.5)--(7.5,1.5);
\draw(0.5,2.5)--(7.5,2.5);

		\node[circle] at (1,3){${\color{black}1}$};
\node[circle] at (2,3){${\color{black}1}$}; \node[circle] at
(3,3){${\color{black}1}$}; \node[circle] at (5,3){${\color{black}1}$};
\node[circle] at (6,3){${\color{black}1}$};

		\node[circle,draw=black] at (4,3){$1$};
\node[circle,draw=black] at (7,3){$1$};
		
   	\node[circle,draw=black] at (2,2){$1$}; \node[circle,draw=black]
at (3,2){$2$}; \node[circle,draw=black] at (4,2){$1$};

		\node[circle] at (5,2){${\color{black}2}$};
\node[circle] at (6,2){${\color{black}2}$}; \node[circle] at
(1,2){${\color{black}1}$}; \node[circle] at (7,2){${\color{black}1}$};

		\node[circle,draw=black] at (1,1){$3$};
\node[circle,draw=black] at (3,1){$1$}; \node[circle,draw=black] at
(5,1){$1$}; \node[circle,draw=black] at (7,1){$2$};

		\node[circle] at (2,1){${\color{black}1}$};
\node[circle] at (4,1){${\color{black}1}$}; \node[circle] at
(6,1){${\color{black}2}$};
	
	\end{tikzpicture} }

\newcommand{\exE} {
	\begin{tikzpicture} \draw (3,8) node {321}; \draw (0,6) node
{231}; \draw (6,6) node {312}; \draw (0,2) node {213}; \draw (6,2)
node {132}; \draw (3,0) node {123};

	\draw node at (1.5,7.5){$1/t_2$}; \draw node at
(4.6,7.3){$1/t_1$}; \draw node at (3.4,6.5){$1/t_1$};

	\draw node at (-0.5,4.5){$1/t_1$}; \draw node at
(6.6,4.5){$1/t_1$};

	\draw node at (1.5,4.5){$1/t_1$}; \draw node at
(4.6,4.5){$1/t_2$};

	\draw node at (1.5,1.3){$1/t_1$}; \draw node at
(4.3,1.3){$1/t_2$};

	\draw [->, >=stealth, thick] (2.7,7.7) -- (0.3,6.3);
\draw [->, >=stealth, thick] (3.3,7.7) --
(5.7,6.3); \draw [->, >=stealth, thick] (0,5.7) -- (0,2.3);
\draw [<-, >=stealth, thick] (2.7,0.3) --
(0.3,1.7); \draw [<-, >=stealth, thick] (3.3,0.3) -- (5.7,1.7);
\draw [<-, >=stealth, thick] (6,2.3) --
(6,5.7); \draw [->, >=stealth, thick] (3,0.3) -- (3,7.7);
\draw [->, >=stealth, thick] (5.7,2.3) --
(0.3,5.7); \draw [->, >=stealth, thick] (0.3,2.3) -- (5.7,5.7);
\draw (4.1,7.9) node
[color=blue] {$t_1 / Z$}; \draw (-0.7,6.7) node [color=blue] {$(t_1+t_2)
/ Z$}; \draw (6.7,6.7) node [color=blue] {$(t_1+t_2) / Z$}; \draw
(-0.7,1.3) node [color=blue] {$t_1$ / Z}; \draw (6.7,1.3) node
[color=blue] {$t_1 / Z$}; \draw (4.5,-0.1) node [color=blue] {$(t_1+t_2)
/ Z$};
\end{tikzpicture} }

\begin{document}

\title{A product formula for the TASEP on a ring} \author{Erik Aas
\and Jonas Sj{\"o}strand} \address{Department of Mathematics, Royal
Institute of Technology \\ SE-100 44 Stockholm, Sweden}
\email{eaas@kth.se} \email{jonass@kth.se} \date{Sept 2013}
\keywords{TASEP, exclusion process, multi-line queue, Markov chain}
\subjclass[2010]{60C05, 82B20}

\begin{abstract}
For a random permutation sampled from the stationary distribution
of the TASEP on a ring,
we show that, conditioned on the event that the first entries
are strictly larger than the last entries, the \emph{order}
of the first entries is independent of the \emph{order} of the
last entries. The proof uses multi-line queues as defined by
Ferrari and Martin, and the theorem has an enumerative combinatorial
interpretation in that setting.

Finally, we present a conjecture for the case where the small and large
entries are not separated.
\end{abstract}

\maketitle

\section{Introduction}
Exclusion processes are Markov chains that can be defined for any graph; see
\cite{liggett} for an overview. The
vertices of the graph are called \emph{sites}, and a state in the
process is a distribution of particles into the sites with \emph{at
most one particle} at each site---this is what the word ``exclusion''
refers to. The transitions in the Markov chain are defined by letting
the particles jump or swap randomly according to some rules. The most
studied case is probably the \emph{totally asymmetric simple exclusion process
(TASEP)} on $\Z$, the set of integers
thought of as an infinite row of sites.  Each particle goes one step
to the left with rate one, but only if that site is empty, otherwise
nothing happens.  The ``totally asymmetric'' part refers to particles
going left but never right.

Recently, several authors (\cite{ferrari-martin}, \cite{arita-mallick},
\cite{ayyer-linusson}, \cite{lam-williams}) have examined the TASEP on
a ring $\Z/n\Z$, where furthermore particles have different \emph{sizes}
(or \emph{classes}). 
This is sometimes called the \emph{multi-type TASEP}. Lam
\cite{lam} gave a nice connection between this TASEP and the limit
shape of large random $n$-core partitions.

The TASEP on the ring $\Z/n\Z$ can be described as a continuous-time
Markov chain defined on permutations of a finite multiset of positive integers.
Given a permutation $w$ written as a word $w=w_1w_2\dotsm w_n$, each
entry $w_j$ carries an exponential clock that rings with rate
$x_{w_j}$.  When it rings, the entry trades place with its left
neighbour (cyclically) if that entry is larger than $w_j$.  Otherwise,
nothing happens.

For a random permutation $w$ sampled from the stationary distribution,
we show that, conditioned on the event that the first $k$ entries of
$w$ are strictly larger than the last $n-k$ entries, the \emph{order}
of the first $k$ entries is independent of the \emph{order} of the
last $n-k$ entries. We also generalize this result to three or more
groups of adjacent entries.
The proof is based on a connection between the
TASEP and multi-line queues which was found by Ferrari and
Martin~\cite{ferrari-martin}.

The paper is organized as follows. After introducing some
notation and presenting our results, in Section~\ref{se:ferrari}
we describe the multi-line queues of Ferrari and Martin and their
connection to TASEPs. In the two
succeding sections we prove two recurrence relations
that will be the main
ingredients in our proof. Section~\ref{sec:proof} contains
the proof of our main result and in Section~\ref{sec:explicit} we
give some explicit formulas for the coefficients in
the recurrence relations.
Finally, in Section~\ref{sec:future} we suggest some future research
and offer a conjecture for the case where the small and large entries
in the word are not separated.

\section{Notation and results}
Fix a positive integer $n$.  We think of
the $n$-ring $\Z/n\Z$ as a row of $n$ \emph{sites}, indexed by
$1,2,\dotsc,n$ from the left, where the left and right edges of the
row are glued together.  The states in the Markov chain are identified
with words $w$ in $\{1, 2, \dots, \infty\}^n$.  Let $m_i$ be the
number of times $i$ occurs in $w$.  We will consider only words $w$
such that $m_1, \dots, m_r > 0$ for some
$r$ and $m_{i}= 0$ for $r<i<\infty$. 
Such a sequence $m_1,\dots,m_r$ will be called a \emph{type}
and we say that $\mbf{m} = (m_1,\dots,m_r)$ is the \emph{type of $w$}.
Note that the type determines $m_\infty$, as the length $n$ is fixed.
The finite entries in $w$ should be interpreted as particles of the
corresponding \emph{size} and the infinite entries
should be interpreted as \emph{empty} sites.  Let $\Omega_{\mbf m}$
denote the set of words of type $\mbf m$ (and length $n$ that
is to be understood from the context).

\begin{defi} Let $\mbf m=(m_1,\dotsc,m_r)$ be a type and let
$t_1,\dotsc,t_r$ be positive real numbers.  The \emph{totally
asymmetric exclusion process}, or \emph{TASEP}, on $\Omega_{\mbf m}$
with inverse rates $t_1,\dotsc,t_r$ is defined as follows.

Each finite entry $i$ of $u\in\Omega_{\mbf m}$ carries an exponential
clock that rings with rate $1/t_i$.  When the clock rings, the entry
trades places with its left neighbour if that entry is larger.
\end{defi}
See Figure \ref{fi:omega11} for an example.
For a generic choice of
the $t_i$, this chain has a unique stationary distribution which we
denote by $\pi$.  We will for the most part think of the $t_i$'s as
indeterminates and of the value $\pi(w)$ (for words $w$) as a rational
function of these indeterminates.  Sometimes it is useful to think of
the $\pi(w)$ as actual probabilities.

Of course, we could also let the infinite entries carry
clocks---their jumps would never be successful.  Although this would be a more
uniform definition, it has some drawbacks that make us opt for the
definition above.

Finally, note that the stationary distribution is invariant with
respect to cyclic shifts; for example, we have $\pi(123) = \pi(312) =
\pi(231)$.

\begin{figure}
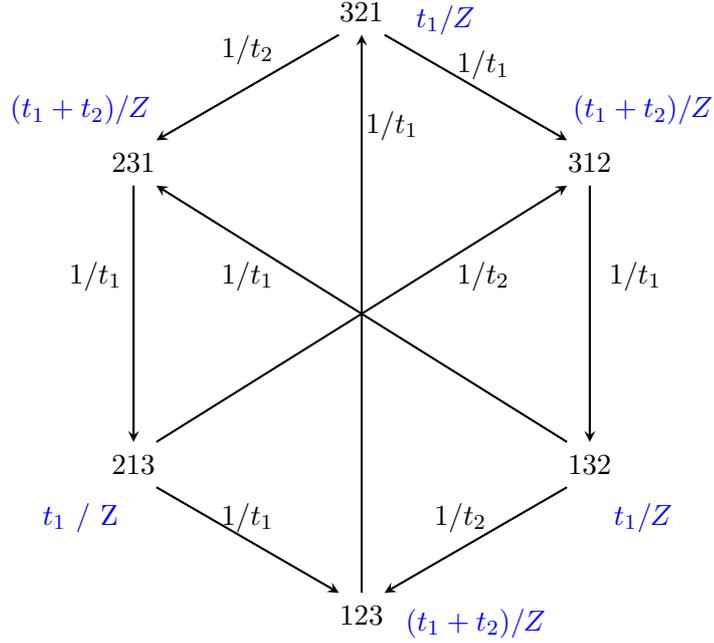

\exE
\caption { The TASEP on $\Omega_{(1,1,1)}$ with $n=3$.
Here $Z = 6t_1+3t_2$.  The stationary distribution $\pi$ is indicated at
each state.  On each arrow is written the corresponding transition
rate.}\label{fi:omega11} 
\end{figure}

A word $w=w_1w_2\dots w_n$ is said to be \emph{decomposable} if it can
be written as a concatenation $w=uv$ of words $u=w_1w_2\dots w_k$ and
$v=w_{k+1}w_{k+2}\dots w_n$ where all entries of $u$ are strictly
larger than all entries of $v$.  (Of course, an infinite entry
$\infty$ is considered to be larger than any finite entry.)

We can now state our main theorem.

\begin{theo}\label{th:main} Let $uv$ be a decomposable word and let
$v'$ be a word obtained from $v$ by permuting the entries.  Then,
\[ \frac{\pi(uv)}{\pi(uv')}=\frac{\pi(\infty v)}{\pi(\infty v')}
\] as rational expressions in the $t_i$'s.
(Here, $\infty$ denotes the word $\infty\dotsb\infty$ of the same length
as $u$.)
\end{theo}
We postpone the proof until Section~\ref{sec:proof} and turn our interest
to some corollaries for now.

Since the right hand side of the equation above is independent of $u$
(but dependent of the length of $u$), the following weaker but
symmetric variant of our main result follows.
\begin{coro}\label{co:firstmaincoro} Let $uv$ be a decomposable word and let
  $u'$ and $v'$ be words obtained from $u$ and $v$, respectively, by
  permuting the entries.  Then,
\[
\pi(uv)\pi(u'v') = \pi(u'v) \pi(uv')
\]
as rational expressions in the $t_i$'s.
\end{coro}

As a further corollary we obtain the result
that we promised in the introduction.
\begin{coro}\label{co:independence}
  Let $uv$ be a decomposable word of type $\mbf m$ and let $W$ be a
  random word sampled from the stationary distribution of the TASEP on
  $\Omega_{\mbf m}$, conditioned on the event that $W=UV$ is
  decomposable with $U$ and $V$ of the same length (and type) as $u$
  and $v$, respectively. Then, $U$ and $V$ are independent, that is
\[
P(U=u\ \text{and}\ V=v)=P(U=u)P(V=v).
\]
\end{coro}
\begin{proof}
By definition,
\begin{align*}
P(U=u\ \text{and}\ U=v) &= \pi(uv)/S, \\
P(U=u) &= \sum_{v'}\pi(uv')/S,\ \text{and} \\
P(V=v) &= \sum_{u'}\pi(u'v)/S,\ \text{where} \\
S &= \sum_{u',v'}\pi(u'v')
\end{align*}
and where the sums are taken over all words $u'$ and $v'$ of the same type
and length as $u$ and $v$,
respectively. From Corollary~\ref{co:firstmaincoro} we obtain
\[
\pi(uv)\sum_{u',v'}\pi(u'v')=\sum_{v'}\pi(uv')\sum_{u'}\pi(u'v)
\]
and we are done.
\end{proof}

As an example, consider the TASEP on $\Omega_{(1,1,1,1)}$ with
$n=4$ and $t_1=t_2=t_3=t_4=1$. 
Our result tells us that if we sample a word $w$
from the stationary distribution and it happens that the entries $3$
and $4$ comes before $1$ and $2$, then the order of the $3$ and $4$
gives us no information about the order of the $1$ and $2$. This can
be stated in equational form as $\pi(4321)/\pi(4312) =
\pi(3421)/\pi(3412)$, which is true since $\pi(4321)=1/96$,
$\pi(4312)=\pi(3421)=3/96$, and $\pi(3412)=9/96$ (as a simple
computation yields).  At first glance our result may seem obvious,
at least in the homogeneous case where all $t_i=1$, since the ``large'' entries
cannot tell the ``small'' entries apart and vice versa. However, this
is an illusional simplicity and the condition that the large numbers
are adjacent is essential for our main result. Indeed, if it happens
that the entries $3$ and $4$ are at positions $2$ and $4$, then their
order really tells us something about the order of the entries $1$ and
$2$, since
\begin{equation}
\label{eq:neq}
3/5=\pi(2413)/\pi(1423)\ne\pi(2314)/\pi(1324)=5/3.
\end{equation}

Finally, the result can be generalized to words that are decomposable
into three or more parts. Let us say that a word is
\emph{$k$-decomposable} if it can be written as a concatenation
$u^{(1)}u^{(2)}\dotsb u^{(k)}$ of $k$ words, where all entries of
$u^{(1)}$ are strictly larger than all entries of $u^{(2)}$, which in
turn are strictly larger than all entries of $u^{(3)}$ and so on.
\begin{coro}
  Let $u^{(1)}u^{(2)}\dotsb u^{(k)}$ be a $k$-decomposable word of
  type $\mbf m$ and let $W$ be a random word sampled from the
  stationary distribution of the TASEP on $\Omega_{\mbf m}$,
  conditioned on the event that $W=U^{(1)}U^{(2)}\dotsb U^{(k)}$ is
  $k$-decomposable with $U^{(i)}$ of the same length (and type) as
  $u^{(i)}$ for each i.  Then, the $U^{(i)}$ are independent, that is
\[
P(U^{(i)}=u^{(i)}\ \forall i)=\prod_{i=1}^k P(U^{(i)}=u^{(i)}).
\]
\end{coro}
\begin{proof}
We know from Corollary~\ref{co:independence} that, for each $i$, the two tuples
$(U^{(1)},\dotsc,U^{(i)})$ and $(U^{(i+1)},\dotsc,U^{(k)})$ are independent,
and it follows that each word $U^{(i)}$ is independent of the tuple
$(U^{(i+1)},\dotsc,U^{(k)})$ of subwords to its right. But this implies that
\begin{multline*}
  P(U^{(i)}=u^{(i)}\ \forall i)=P(U^{(1)}=u^{(1)})P(U^{(i)}=u^{(i)}\ \forall i\ge2)\\
  =P(U^{(1)}=u^{(1)})P(U^{(2)}=u^{(2)})P(U^{(i)}=u^{(i)}\ \forall
  i\ge3)=\dotsb =\prod_{i=1}^k P(U^{(i)}=u^{(i)}).
\end{multline*}
\end{proof}

\section{Queues and multi-line queues}
\label{se:ferrari}
Ferrari and Martin~\cite{ferrari-martin} studied the \emph{homogenous}
case where $t_1=\dotsb=t_r$ and
discovered that the TASEP on a ring
is a projection of a richer process involving combinatorial structures called
multi-line queues. Ayyer and Linusson~\cite{ayyer-linusson} generalized
the construction to the inhomogeneous case. Remarkably,
this transfers questions about the stationary distribution
of the TASEP to enumerative questions about multi-line queues and we state
that result in Theorem~\ref{th:mlq}. But first we must define
queues and multi-line queues.

\subsection{Queues}\label{sec:queues}
A \emph{queue} $q$ of capacity $c$ is an $n$-ring where $c$
sites are marked as \emph{terminal}.  It defines a map $q$ on words,
defined as follows.

Given an input word $u$ of type $(m_1,\dotsc,m_r)$, first replace
each letter $\infty$ (empty site) by $r+1$ (particle of a new maximal
size).  Next, put the queue immediately below $u$ to form a $2\times
n$-array of sites (see Figure \ref{fi:single_queue}), and perform the
following queuing procedure.

We will associate a word $q(u)$ to $q$ whose empty sites are exactly at the
non-terminal sites of $q$.

At the start, all terminal positions of $q$ are \emph{unoccupied}.
\begin{itemize}
\item We will let the $n$ input particles enter the queue one by one
in any order such that smaller particles come before larger ones.
(The order of particles of the same size does not matter.)
\item According to such an ordering, each particle first enters the
queue by moving one step down, and then it moves zero or more steps
(cyclically) to the right until it reaches an unoccupied terminal
site.  If all terminal sites are occupied, the particle just continues
around the ring forever and can be discarded.  (However, it may have
an impact on the \emph{weight} of $q$ which is defined below.)
\end{itemize} The resulting word is the output of the map.

Note that each non-terminal site in $q$ is visited by at least one
particle.  For $i=1,2,\dotsc,r$, let $\alpha_i$ be the number of
non-terminal sites in the queue that are first visited by a particle
of size $i$.

The monomial $t_1^{\alpha_1}\dotsm t_r^{\alpha_r}$ is
called the \emph{weight of $q$ with respect to $u$}.

Note that neither the weight nor the resulting word depends on
the order in which particles of the same size enter the queue.

\begin{figure}
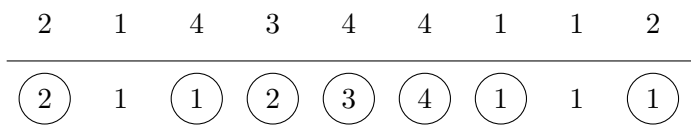

\exC
\caption{A $2\times n$ array representing a queue with input word
$21\infty 3\infty\infty 112$ (converted to $214344112$) and output
word $2\infty 12341\infty 1$.  The circles on the bottom row
corresponds to the terminal positions of the queue.  The two non-circled
1's in
the bottom row corresponds to non-terminal positions visited by $1$.
Thus the weight of this queue with respect to this input is $t_1^2$.}
\label{fi:single_queue} 
\end{figure}

\subsection{Multi-line queues}
A \emph{multi-line queue}
of type $\mbf m=(m_1,\dotsc,m_r)$ is
a sequence $\mbf q=(q_1,\dotsc,q_r)$ of queues such that $q_i$ has
capacity $m_1+\dots+m_i$ for each $i$.  We will think of the
queues as being stacked vertically and view the multi-line queue as an
$r\times n$-array of sites where exactly $m_1+\dots+m_i$ of the
sites in row $i$ are marked as \emph{terminal}.

Let us feed the first queue $q_1$ with the ``empty'' word
$\infty\dots\infty$ and then feed $q_2$ with the output from $q_1$ and
$q_3$ with the output from $q_2$ and so on, see Figure
\ref{fi:multi_queue}. (We will denote an empty word simply by $\infty$,
the length of the word usually being clear from the context.)  The resulting
word $w(\mbf q)=(q_r\circ\dots\circ q_2\circ q_1)(\infty)$ is called
the \emph{output word} of $\mbf q$ and has type $\mbf m$.

The \emph{weight} of $\mbf q$, denoted by $[\mbf q]$, is a monomial in
the variables $t_1,\dots,t_r$ defined as the product of the
weights of $q_i$ with respect to $(q_{i-1}\circ\dots\circ q_2\circ
q_1)(\infty)$ for $1\le i\le r$.  Also, for any word $w$, define its
\emph{weight} $[w]$ as the polynomial in $t_1,\dots,t_r$ obtained
by adding the weights of all multi-line queues $\mbf q$ with output
$w(\mbf q)=w$.  Note that letting all $t_i = 1$ in the polynomial
$[w]$ gives the number of such multi-line queues.  Finally, let
$Z_{\mbf m}=Z_{\mbf m}(t_1,\dots,t_r)$ denote the sum of the
weights of all multi-line queues of type $\mbf m=(m_1,\dots,m_r)$.
(Recall that $n$ is fixed.)

\begin{figure}
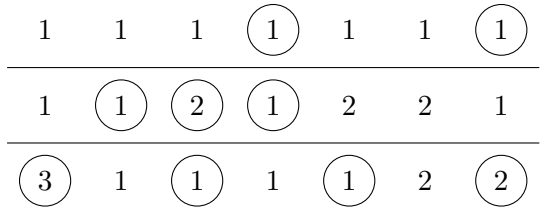

\exD
\caption{A multi-line queue. The circles correspond to terminal
positions in the queues. The output of this multi-line queue is
$3\infty 1 \infty 1 \infty 2$. A number $i$ in a non-terminal position
corresponds to a particle $i$ having visited that position first. Thus the
weight of this multi-line queue is $t_1^9 t_2^3$.}
\label{fi:multi_queue}
\end{figure}

The following theorem has appeared in increasingly stronger versions
in \cite{angel}, \cite{ferrari-martin} and \cite{ayyer-linusson}.  In
\cite{ayyer-linusson} the strongest version (given below) is stated as
a conjecture.  The conjecture has recently been proved in at least two
ways \cite{arita-mallick}, \cite{linusson-martin}.

\begin{theo}
\label{th:mlq}
The stationary distribution $\pi$ of the TASEP on
$\Omega_{\mbf m}$ with inverse rates $t_1,\dots,t_r$ is given by
$\pi(w)=[w]/Z_{\mbf m}$.
\end{theo}

At this point, the reader may wish to
skip directly to the proof of Theorem \ref{th:main} in Section
\ref{sec:proof}, and then read the results in the intermediate
sections (on which the proof depends).  In the intermediate sections
we will, for each word $u$, express $[u]$ in terms of various $[v]$'s,
where the $v$'s are \emph{simpler} than $u$.

\begin{defi}
A type $(m_1, \dots, m_s)$ is \emph{simpler than} another
type $(m'_1, \dots, m'_r)$ if $s<r$, or $s=r$ and $m_s > m'_r$.  Let
$u, v$ be two words of the same length.  We say that $v$ is \emph{simpler
than} $u$ if either
\begin{itemize}
\item $v$ has simpler type than $u$, or
\item $u$ and $v$ have the same type, and if $i$, $j$ are the smallest
numbers such that $u_i = \infty$ and $v_j = \infty$, then $i > j$.
\end{itemize}
\end{defi}

\section{The first lemma}

For any word $w$ of type $\mbf{m}=(m_1,\dots,m_r)$, let $w^-$ denote
the word obtained from $w$ by removing all particles of maximal size
(that is, replacing all occurrences of $r$ in $w$ by $\infty$).  The
type of $w^-$ is clearly $(m_1,\dots,m_{r-1})$, which we denote
by $\mbf{m}^-$.  The following easy and well-known lemma will give us
our first recursion.

\begin{lemma}
\label{lm:prob} For any word $u$ of type $\mbf m^-$, we have $\sum_w
\pi(w)=\pi(u)$, where the sum is over all words $w$ of type $\mbf m$
such that $w^-=u$.
\end{lemma}
\begin{remark}\label{re:trivial}
In the special case where $m_1+\dotsb+m_r=n$ the lemma
simply states the obvious fact that, given a word with all entries
finite, we can disregard the last variable
$t_r$ and replace all occurrences
of $r$ by $\infty$ without altering the stationary probability.
(In fact this does not alter the value of the bracket either since
$Z_{\mbf m}=Z_{\mbf m^-}$ in this case; see Theorem~\ref{th:mlq} and
Corollary~\ref{co:explicitZ}.)
\end{remark}

Translating Lemma \ref{lm:prob} using Theorem \ref{th:mlq} and
specializing to decomposable words (for later convenience), we get the
following.
\begin{lemma}
\label{lm:ET} Suppose that $uv$ is a decomposable word of type
$\mbf{m}$ and that all maximal particles precede all empty sites in $u$. Then
	\[ [uv] = \frac{Z_{\mbf{m}}}{Z_{\mbf{m}^-}}[u^- v] - \sum_{u'}
[u'v],
	\] summing over all $u' \neq u$ obtained from $u$ by permuting
the particles of maximal size and empty sites. Furthermore, all the words
occurring on the right hand side are decomposable and simpler than
$uv$. 
\end{lemma}
\begin{proof} The word $u^-v$ is simpler than $uv$ since it has
simpler type, and each $u'v$ is simpler than $uv$ since among all
rearrangements of empty sites and maximal particles in $u$, $u$ is the
one where the first empty site comes latest, by the assumption.
\end{proof}

For example, if $u|v = 3434\infty 3|21221$
(the delimiter is for emphasis only), we get the formula $ [3434
\infty 3|21221]= \frac{Z_{2,3,3,2}}{Z_{2,3,3}} [3 \infty 3 \infty
\infty 3|21221]-[343\infty 43|21221]-[3\infty
3443|21221]$,
where each term on the right hand side is simpler than $uv$.
(Using Lemma \ref{co:explicitZ}, we have
$\frac{Z_{2,3,3,2}}{Z_{2,3,3}} = h_{11-(2+3+3+2)}(t_1,t_1,
t_2,t_2,t_2,t_3,t_3,t_3,t_4,t_4,t_4) = t_1+3t_2+3t_3+3t_4$, though the
explicit form of this factor is not important for the proof.)

Lemma \ref{lm:ET} is not very useful by itself. It will, however, work
in those (few) cases where Lemma \ref{lm:JT} below fails.

\section{The second lemma}

Our goal in this section is to find a more general formula for relating
$[u]$ to $[v]$ for words $v$ simpler than $u$.
A natural way to find such a formula is to consider a pair $(\mbf{p},
\mbf{q})$ of multi-line queues differing only in their last respective
queues $p$ and $q$. Then the bottom row of one of them will be
simpler than the bottom row of the other. (In fact, the relation between
them will be a covering relation in the partial order given by
``simpler than''.)

We will assume that the first site in $p$ is terminal while the first
site in $q$ is non-terminal.  The multi-line queues $\mbf{p}$ and
$\mbf{q}$ will be otherwise identical.

Let us start by looking at an example.
Consider the queue of length 10 with non-terminal sites 4 and 9.
In Figure~\ref{fi:ex4.3}
we feed it with the
word $w' = 21\infty 2\infty\infty 3\infty 31$ and obtain
the output word $v = 213\infty 2434 \infty 1$. What would happen if
the first site of the queue were non-terminal instead?

The 2-particle that settled at site 1 before will continue to the right
until it reaches a non-occupied terminal site.
That site could not be site 2 because it is
already occupied by a 1-particle (since 1 is less than 2).
But the 3-particle that will try to occupy
site 3 has not yet arrived so the 2-particle will stop there. Later,
when the 3-particle arrives it finds that site 3 is already occupied so it
continues to the right and stops at site 6. Later, when the 4-particle
arrives at site 6 it finds it already occupied and continues to the right.
Either of the two 4-particles will be unable to find an unoccupied
terminal position and the process comes to an end. The output word
will be $w = \infty12\infty 2334 \infty 1$ and can be obtained from
$u$ by the operation $w=v^{1\ra}$ defined as follows.

\begin{defi}
\label{de:jump}
Let $u$ be a word with at least two maximal particles
and let $j_0$ be an index such that $u_{j_0} \neq \infty$.
Define a sequence $j_1,\dotsc,j_s$ called the
\emph{jumping sequence from $j_0$}
recursively by letting
$j_k$ be the first position cyclically to the right of
$j_{k-1}$ such that $u_{j_k}$ is finite and strictly larger than
$u_{j_{k-1}}$. If no such position exists, we define $s = k-1$.

Now
define a word $v = u^{j_0\ra}$ by letting $v_{j_k} = u_{j_{k-1}}$ for
$1 \leq k \leq s$ and $v_{j_0} = \infty$.  The words $v$ and $u$ agree
in all other positions. 
\end{defi}
For example, $155231 \infty 42^{4\ra} = 145\infty 21\infty 32$ and
$11^{1\ra} = \infty1$.

We stress the remarkable fact (for now just an observation in this
particular example) that $w$ is determined by $v$ (and independent of
$w'$).
However it is not clear what the weight of $[q]$ is in terms
of $[p]$; this requires knowing more about $w'$.

Let us turn this relation around and for a fixed $w$ as above ask
which $v$ are possible?  It happens that there are $18$ such $v$'s so
let us do this for a smaller example.
Suppose instead that $w = \infty 3412\infty 31$.  Then there are
two $v$'s satisfying $v^{1\ra} = w$: $43412\infty 31$ and $34412\infty
31$.  (In Lemma~\ref{lm:laexplicit} we make explicit how to find these
words though we will not need it for the proof of Theorem~\ref{th:main}.)
Furthermore, in this smaller example, we can actually determine the
weight of $[q]$ in terms of $[p]$: If $v = 43412\infty 31$ then
$[\mbf{q}] = t_4[\mbf{p}]$, and if $v = 34412\infty 31$ then
$[\mbf{q}] = t_3[\mbf{p}]$ (independently of $w'$).

As we shall see later, the property of $w$ that allows us to read off
$[\mbf{p}]$ in terms of $[\mbf{q}]$ is that the first maximal particle
in $w$ to the right of the empty site $1$ comes before any additional
empty site.

Now, the weight of all multi-line queues with bottom row $v$ is counted by
$[v]$ on one hand and by $t_4[43412\infty 31] + t_3[34412\infty 31]$
on the other hand, so
\[ [\infty 3412\infty 31] = t_4[43412\infty 31] + t_3[34412\infty 31].
\]
(The number of multi-line queues representing $\infty 3412\infty 31$,
$43412\infty 31$ and $34412\infty 31$ are $624,136$ and $488$; we have
$624 = 136 + 488$ as predicted by the formula above when we let all
$t_i$ equal $1$.)

\begin{figure}
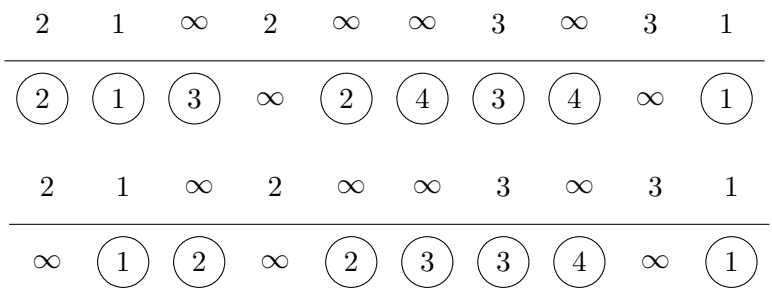

\exA \vspace{0.5cm}

\exB
\caption{The last two lines of two multi-line queues that are
identical except the for site $1$ in the last row.}
\label{fi:ex4.3}
\end{figure}

In the next lemma we give the general statement corresponding to the
examples above.

\begin{lemma}
\label{lm:bijection} Let $(\mbf p,\mbf q)$ be a pair of multi-line
queues that are identical except that site $j_0$ of the last queue is
terminal in $\mbf p$ and non-terminal in $\mbf q$.
Then, \begin{enumerate}
\item[(a)] $w(\mbf q)=w(\mbf p)^{j_0\ra}$, and
\item[(b)] if, starting from site $j_0$ in $w(\mbf p)$ and going
cyclically to the right, the first maximal particle comes before any
empty site, then $[\mbf q]=t_k[\mbf p]$, where $k$ is
the size of the particle at site $j_0$ in $w(\mbf p)$.
\end{enumerate}
\end{lemma}
\begin{proof} We may assume without loss of generality that
$j_0=1$.

Let $w'=(p_{r-1}\circ\dots\circ p_2\circ
p_1)(\infty)=(q_{r-1}\circ\dots\circ q_2\circ q_1)(\infty)$ be the
input permutation to the final queue $p:=p_r$ in $\mbf p$ and to the
final queue $q:=q_r$ in $\mbf q$.  Let us perform the queueing
procedure as defined in Section~\ref{sec:queues} for the queues $p$
and $q$ in parallel and with the same ordering of the input particles
(we need not consider all of $\mbf p$ or $\mbf q$ anymore).

We will pause this procedure at some moments in time for a close
examination, and these moments will be indexed by $k=0,1,2,\dots$ just
for bookkeeping.


Let us say that a site in a queue is \emph{open} if it is terminal and
unoccupied (so from the beginning all terminal sites are open).  As we
will see by induction, at each moment $k$, the queues $p$ and $q$ have
the same set of open sites, except for one site $j_k$ which is open in
$p$ but not in $q$.  From the beginning, at moment 0, this special
site is the first one, $j_0=1$.

For brevity, let us put $v:=w(\mbf p)$ and $w:=w(\mbf q)$.  The
induction step goes as follows.
\begin{itemize}
\item When we resume the queuing procedure after moment $k$, the
particles will move identically in $p$ and $q$ until the site $j_k$ is
visited by some particle for the first time in $p$.  In $p$ the
particle will stop at $j_k$ so it has size $v_{j_k}$, but in $q$ site
$j_k$ is not open, so here the corresponding particle will continue to
the right until it reaches the first open site $j_{k+1}$ to the right
of $j_k$, and hence $w_{j_{k+1}}=v_{j_k}$.  At this time we pause
again and call it moment $k+1$.
\item However, it may happen that the particle never reaches an open
site in $q$.  In this case, $v_{j_k}=r$ and the remaining part of the
queuing procedure will be perfectly identical for $p$ and $q$.  There
is no need for further pauses.
\end{itemize}

During the induction step above, before the $v_{j_k}$-particle stops
at site $j_{k+1}$ in $q$ it passes all terminal sites between $j_k$
(inclusive) and $j_{k+1}$ (exclusive), so these are already occupied.

Hence $v_{j_k}$ is greater than or equal to all finite entries strictly
between $v_{j_k}$ and $v_{j_{k+1}}$ but it is smaller than or equal to
$v_{j_{k+1}}$. It follows that $w=v^{1\ra}$, showing (a).

To show (b), suppose that no empty site of $v$ precedes the first
$r$-particle, that is, there is no non-terminal site to the left of
site $j_s$ in $p$.  Then the first visiting particle of each
non-terminal site of $p$ is also the first visitor of the same
non-terminal site in $q$.  But the first site is terminal in $p$ but
non-terminal in $q$, and this site is first visited by a particle of
size $v_1$.  We conclude that $[\mbf p]=t_{v_1}[\mbf q]$.
\end{proof}

Fix a word $w$ with $w_1 = \infty$.  Note that, by the previous lemma,
the sets $\{(\mbf{p}, \mbf{q}) : w(\mbf{p})^{1\ra} = w\}$ and
$\{(\mbf{p}, \mbf{q}) : w(\mbf{q}) = w\}$ (where $\mbf{p}$ and
$\mbf{q}$ differ in the first site of the last row only, as earlier)
are equal.  This proves the following lemma.

\begin{lemma}
	\label{lm:rec2} Suppose $u$ is a word such that site $j_0$ is
empty and, going
cyclically to the right from $j_0$, the first maximal particle
comes before any additional empty site.
Then
		\[ [u] = \sum_v t_{v_{j_0}}[v],
		\] summing over all $v$ with $v^{j_0\ra} = u$.
\end{lemma}

We now refine Lemma \ref{lm:rec2} slightly for the case of
decomposable words, for later convenience.
\begin{lemma}
\label{lm:JT} Suppose $uv$ is decomposable, that site $j_0$ in $u$ is
empty, and that there is a site $j > j_0$ in $u$ with a maximal particle
such that there are no empty sites or maximal particles
strictly between $j_0$ and $j$ in $u$.

	Then
	\[ [uv] = \sum t_{u'_{j_0}}[u'v],
	\] where the sum extends over all $u'$ such that $u'^{j_0\ra}
= u$.  Moreover, the words $u'v$ occurring on the right hand side
are decomposable and simpler
than $uv$.
\end{lemma}
\begin{proof} We need to prove that a word $w$ satisfies
$w^{j_0\ra}=uv$ if and only if $w$ is a decomposable word of the form
$w=u'v$ where $u'^{j_0\ra}=u$.
This is clear since whenever $w^{j_0\ra}=u$, the associated
jump sequence necessarily ends before or at $j$.
Since $u'_i \geq u_i$ for each $i$, it follows that $u'v$ is
decomposable.  Each $u'v$ is simpler than $uv$ since it has simpler
type than $uv$.
\end{proof}

\section{Finishing the proof}
\label{sec:proof}
\noindent

\begin{proof}[Proof of Theorem \ref{th:main}]

Fix $u_0, v_0, v'_0$ such that $u_0v_0$ is decomposable and $v'_0$ has
the same type as $v_0$.  By Theorem~\ref{th:mlq} we can work with brackets
instead of probabilities, and we want to prove that
$\frac{[u_0v_0]}{[u_0v'_0]} = \frac{[\infty v_0]}{[\infty v'_0]}$.

In the course of the proof we will consider decomposable words $uv$ such that $v$
has the same type as $v_0$ and $uv$ is simpler
than $u_0v_0$.

Our main goal then will be to derive, for each such $uv$, an
identity of the form
\begin{equation} \label{ppp} [uv] = f(u) [\infty v],
\end{equation} where $f(u)$ is some polynomial in the $t_i$'s. The exact form
of $f(u)$ is not important, only the fact that it depends only on $u$
and not on $v$.  To see that this proves the theorem, note that
$[u_0v'_0] = f(u_0) [\infty v'_0]$, which implies the desired
conclusion $\frac{[u_0v_0]}{[\infty v_0]} = f(u_0) =
\frac{[u_0v'_0]}{[\infty v'_0]}$, or
$\frac{[u_0v_0]}{[u_0v'_0]}=\frac{[\infty v]}{[\infty v'_0]}$.

To prove (\ref{ppp}), we will for each $uv$ prove an identity of the
form
\begin{equation}
	\label{ppprec} [uv] = c_1[u_1'v] + c_2[u_2'v] + \dotsb,
	\end{equation} where all the $u'_i$ are words simpler than
$u$. The coefficients $c_i$ are polynomials in the $t_i$'s
(actually, for the most they will be simply $1$) and they are independent of $v$;
recall that we only
consider $v$ of same type as the fixed word $v_0$.

Note that there is a unique simplest word of the form $uv_0$ among those simpler than
$u_0v_0$, namely $\infty v_0$. Hence, by applying (\ref{ppprec}) recursively
we obtain an identity of the desired form (\ref{ppp}).

So it remains to state and prove an identity of the form
(\ref{ppprec}) for each $uv$.  The identity will come in two forms,
depending on the exact structure of $u$.  They are given exactly by
Lemmas \ref{lm:ET} and \ref{lm:JT}; for a given $u$ exactly one of
these apply!

\end{proof}

\subsubsection*{Example}

Let $v$ be any word of the same type as $122$. We give an example of the
argument in the proof of the theorem for the decomposable word
$534v$.

First we use Lemmas \ref{lm:ET} and \ref{lm:JT} to generate a set of
identities where the words on the right hand side are simpler than the
word on the left hand side.  The identities with a term involving
$Z_{\mbf{m}}/Z_{\mbf{m}^-}$ are generated by Lemma~\ref{lm:ET}, the
others from Lemma~\ref{lm:JT}.  (Here we write $r+1$ instead of
$\infty$ for words of type $\mbf{m} = (m_1,\dots,m_r)$;
see Remark~\ref{re:trivial}.)

$ [534v] = t_4[434v] + t_3[344v] $

$ [434v] = t_3[334v] $

$ [344v] = \frac{Z_{1,2,1}}{Z_{1,2}}[333v] - [434v] - [443v] $

$ [334v] = \frac{Z_{1,2,2}}{Z_{1,2}}[333v] - [433v] - [343v] $

$ [434v] = t_3[334v] $

$ [443v] = t_3[433v] = t_3^2[333v] $

$ [433v] = t_3[333v] $

$ [343v] = t_3[333v] $

Starting from the first equation and then repeatedly subsituting to
express $[534v]$ in yet simpler words we get $[534v] =
(\frac{Z_{1,2,2}}{Z_{1,2}}(t_3t_4-t_3^2)+\frac{Z_{1,2,1}}{Z_{1,2}}t_3-2t_3^2t_4-t_3^2-2t_3^3)[333v]$.
Since the parenthesized expression depends only on the type of $v$ we get for
example that $[534212] / [333212] = [534221] / [333221]$.

\section{Explicit formulas}
\label{sec:explicit}

In this section, we make the relations given in Lemmas \ref{lm:ET} and
\ref{lm:JT} more explicit.

\begin{lemma}\label{lm:Zforsinglequeue} Let $u$ be any word of type
$\mbf m=(m_1,\dotsc,m_{i-1})$ and let $m_i$ be a positive integer.
The sum of the weights with respect to $u$ of all queues of capacity
$c=m_1+\dotsb+m_i$ is given by
\begin{equation}\label{eq:Zforsinglequeue} h_{n-c}
(\underbrace{t_1,\dotsc,t_1}_{m_1},
\underbrace{t_2,\dotsc,t_2}_{m_2},\dotsc,
\underbrace{t_{i-1},\dotsc,t_{i-1}}_{m_{i-1}},
\underbrace{t_i,\dotsc,t_i}_{m_i+1}),
\end{equation} where $h_k$ denotes the complete homogeneous symmetric
polynomial of degree $k$.
\end{lemma}
\begin{proof} In the queuing procedure defined in
Section~\ref{sec:queues}, we are free to choose any ordering of the
particles to enter the queue as long as smaller particles precede
larger ones.  Let us fix any such ordering for the particular input
word $u$ (first replacing empty sites by particles of a new maximal
size $i$).  If we are given a queue $q$ of capacity $c$, we can
define its \emph{visiting sequence}
$(\beta_1,\dotsc,\beta_{c+1})$ by letting $\beta_j$ be the
number of sites in $q$ that are first visited by the $j$th particle in
the ordering.  Conversely, if we are given any $(c+1)$-tuple
$\beta=(\beta_1,\dotsc,\beta_{c+1})$ of nonnegative integers
which add up to $n-c$, it is easy to see that there is a unique queue
$q$ with capacity $c$ and visiting sequence $\beta$.

Thus, there is a bijection between queues $q$ of capacity $c$ and
nonnegative integer sequences $(\beta_1,\dotsc,\beta_{c+1})$
adding up to $n-c$, and the weight of $q$ with respect to $u$ is given
by $t_{s_1}^{\beta_1}\dotsm
t_{s_{c+1}}^{\beta_{c+1}}$, where $s_j$ is the size of the $j$th
particle in the ordering.  We conclude that the sum of the weights
with respect to $u$ of all queues with capacity $c$ is
\begin{align*} \sum_{\mathclap{
\substack{(\beta_1,\dotsc,\beta_{c+1})\in\N^{c+1}\\
\beta_1+\dotsc\beta_{c+1}=n-c}}}
t_{s_1}^{\beta_1}\dotsm t_{s_{c+1}}^{\beta_{c+1}}
&=h_{n-c}(t_{s_1},\dotsc,t_{s_{c+1}})\\
&=h_{n-c}
(\underbrace{t_1,\dotsc,t_1}_{m_1},
\underbrace{t_2,\dotsc,t_2}_{m_2},\dotsc,
\underbrace{t_{i-1},\dotsc,t_{i-1}}_{m_{i-1}},
\underbrace{t_i,\dotsc,t_i}_{m_i+1}).
\end{align*}
\end{proof}

We can compute $Z_{\mbf m}$, the sum of the weights of all multi-line
queues of type $\mbf m$, by applying Lemma~\ref{lm:Zforsinglequeue}
multiple times.
\begin{theo}\label{th:Z} The sum of the weights of all multi-line
queues of type $\mbf m=(m_1,\dotsc,m_r)$ is given by
\[ Z_{\mbf m}(t_1,\dotsc,t_r)=\prod_{i=1}^r
h_{n-m_1-m_2-\dotsb-m_i} (\underbrace{t_1,\dotsc,t_1}_{m_1},
\underbrace{t_2,\dotsc,t_2}_{m_2},\dotsc,
\underbrace{t_{i-1},\dotsc,t_{i-1}}_{m_{i-1}},
\underbrace{t_i,\dotsc,t_i}_{m_i+1}).
\]
\end{theo} Theorem~\ref{th:Z} appeared first in \cite{aas}.

Now we can give an explicit formula for the quantity $Z_{\mbf{m}} /
Z_{\mbf{m}^-}$ occurring in Lemma \ref{lm:ET}.

\begin{coro}
\label{co:explicitZ} Let $\mbf{m} = (m_1,\dots,m_r)$ be a type, and
$e = n - m_1 - \dots - m_r$ be the number of empty sites. Then
\[ \frac{Z_{\mbf{m}}}{Z_{\mbf{m}^-}} = h_e
(\underbrace{t_1,\dotsc,t_1}_{m_1},
\underbrace{t_2,\dotsc,t_2}_{m_2},\dotsc,
\underbrace{t_{r-1},\dotsc,t_{r-1}}_{m_{r-1}},
\underbrace{t_r,\dotsc,t_r}_{m_r+1}).
\]
\end{coro}

Finally, we show how to invert the $\ra$ operation, that is,
for a fixed $u$ with $u_1=\infty$, finding all $v$ such that
$v^{1\ra} = u$.

\begin{defi} Suppose $w$ is a word where site $1$ is empty.  For a
subset of sites $J = \{j_1 < \dots < j_s\}$ with $w_j \neq \infty$ for
all $j\in J$, define a word $v = w^{\la J}$ by letting $v_1 =
w_{j_1}$, and $v_{j_k} = w_{j_{k+1}}$ for $0 \leq k < s$.  All other
entries of $v$ and $w$ are the same.
\end{defi}

Note that if $u = v^{1\ra}$ then $v = u^{\la J}$ for \emph{some} $J$
(namely the jump sequence considered as a set).

It is easy to characterize the possible $J$. Essentially, we just have to
check how the inequalities of Definition~\ref{de:jump} among the jumping entries
of $v$ transform into inequalities among entries of $v^{1\ra}$. The result is
given in the following lemma which we present without a formal proof.
\begin{lemma}
\label{lm:laexplicit} Let $u,v$ be two words and suppose $v_1=\infty$.
Then $u^{1\ra} = v$ if
and only if $u = v^{\la J}$ for some $J = \{j_1<\dots<j_s\}$
satisfying the following properties
	\begin{itemize}
		\item $v_{j} \neq \infty$ for each $j\in J$,
		\item for $1 \leq k \le s$, the entry $v_{j_k}$ is
larger than or equal to each preceding finite entry, and
		\item $v_{j_1} < \dots < v_{j_s}<r$, where
$r$ is the maximal particle size in $v$.
	\end{itemize}
\end{lemma}

Another way of saying this is the following.
Let $W$ be the set of positions of non-maximal entries in $v$ that are
greater than or equal to all preceding finite entries.
The sets $J$ described in Lemma \ref{lm:laexplicit} are precisely those
subsets of $W$ such that all $v_j$, $j\in J$ are
distinct. In particular, if the $a_i$ denotes the number of $k$ such
that $v_{j_k} = i$, then the number of possible $J$ is $\prod_i
(1+a_i)$.

\section{Open questions and a conjecture}
\label{sec:future}
Our independence result, Corollary~\ref{co:firstmaincoro},
is a natural probabilistic statement; in fact, as we discussed earlier,
one is easily mislead to believe that it is obvious. However, our proof
is far less natural as it depends on the theory of multi-line queues
and makes use of indirect recurrence relations. It would be desirable
to have a more probabilistic proof not involving multi-line queues but
focusing on the Markov process itself.

As we saw from the inequality~\eqref{eq:neq},
when the large and small entries
are not separated, their orders are correlated. It would be interesting
to quantify this correlation and bound it.
In general, given the positions and the order of the small entries in
a random word sampled from $\pi$, what can be said about the large ones?
As a possible partial answer to that question,
we offer the following conjecture.
\begin{conj}
Let $\mbf m=(1,1,\dotsc,1)$ be the type of the permutation word $12\dotsb n$
and let $W$ be a random permutation sampled from the stationary distribution on
$\Omega_{\mbf m}$. Conditioned on the event that the $k$ smallest entries
in $W$ are in some predefined positions in some predefined order, the
most probable order among the $n-k$ largest entries is cyclically ascending
and their least probable order is cyclically descending.
\end{conj}
For instance, let us say that
I draw a random permutation $W$, tells you that
it happens to be of the form $W={\ast}3{\ast}4{\ast}{\ast}21$ (where the stars are placeholders
for larger entries) and ask you what you think about the order of the
entries $5$ to $8$. Your best guess would be to replace the stars with
a cyclic shift of $5678$, in fact the most probable among those guesses
are $73845621$. Your worst possible guess would be $63548721$.

The conjecture has been checked in the homogeneous case
($t_1=\dotsb=t_r$) by a computer up to $n=10$.



\end{document}